\documentclass[reqno,11pt]{amsart}
\usepackage[utf8]{inputenc}
\usepackage{amsfonts,amsmath,amsthm, url}
\usepackage{amssymb}
\usepackage{graphics}
\usepackage{graphicx}
\usepackage{epsfig,color}
\usepackage{todonotes}
\usepackage{stmaryrd} 
\usepackage{dsfont} 

\newtheorem{prop}{Proposition}
\newtheorem{assumption}{Assumption}
\newtheorem{lem}{Lemma}[section]

\newtheorem{rem}{Remark}[section]
\newtheorem{algo}{Algorithm}[section]

\newcommand{\esp}[1]{\mathbb{E}\left[ #1 \right]}

\newcommand{\prob}[1]{\mathbb{P}\left( #1 \right)}

\newcommand{\R}{\mathbb{R}}

\newcommand{\abs}[1]{\left|#1\right|}

\newcommand{\norm}[1]{\left|\left| #1 \right|\right|}
\newcommand{\norms}[1]{|| #1 ||}
\newcommand{\jac}{\operatorname{Jac}}
\newcommand{\hess}{\operatorname{Hess}}

\renewcommand{\epsilon}{\varepsilon}

\makeatletter
  
  \@addtoreset{equation}{section}
\makeatother
\begin{document}
\title{Generalized and hybrid Metropolis-Hastings overdamped Langevin algorithms}
\author[Romain PONCET]{Romain Poncet$^{\scriptstyle 1}$}

\keywords{Non-reversible diffusions, Langevin samplers, Markov Chain Monte Carlo,
Metropolis-Hastings, variance reduction, lifting method, MCMC}

\maketitle

\begin{center}
$^1$ CMAP, Ecole Polytechnique, CNRS, Université Paris-Saclay, 91128 Palaiseau, France;\\
\email{romain.poncet@cmap.polytechnique.fr}
\end{center}

\vskip 0.1 in
\noindent
{\small
{\bf Abstract}.
It has been shown that the nonreversible overdamped Langevin dynamics enjoy better
convergence properties in terms of spectral gap and asymptotic variance than the
reversible one
(\cite{hwang1993,hwang2005,LelievreNierPavliotis13,Wu2014,ReyBellet15,ReyBellet15_2,
Duncan2016}).
In this article we propose a variance reduction method for the Metropolis-Hastings
Adjusted Langevin Algorithm (MALA) that makes use of the good behaviour of the
these nonreversible dynamics.
It consists in constructing a nonreversible Markov chain (with respect to the target
invariant measure) by using a Generalized Metropolis-Hastings adjustment on a
lifted state space. We present two variations of this method and we discuss
the importance of a well-chosen proposal distribution in terms of average rejection
probability.
We conclude with numerical experimentations to compare our algorithms with
the MALA, and show variance reduction of several order of magnitude in some
favourable toy cases.

\vskip 0.1 in


\section{Introduction}

This article proposes a new class of MCMC algorithms whose objective is to compute
expectations
\begin{align}
  \label{eq:integral}
  \pi(f):=\mathbb{E}_\pi(f)=\displaystyle\int_{\R^d} f(x)\pi(dx),
\end{align}
for a given observable $f$, with respect to a probability measure $\pi(dx)$
absolutely continuous, with respect to the Lebesgue measure,
with density $\pi(x)=e^{-U(x)}$.
We suppose, as it is the case in many practical situations, that $\pi$ is only
known up to a multiplicative constant.
\vspace{2mm}

Many techniques have been developed to solve this problem.
Deterministic quadratures can be very efficient at low dimension.
Yet, in the high dimensional case, these methods tend to become inefficient, and MCMC methods
can be used instead.
The basic idea is to construct an ergodic Markov chain with respect to $\pi$,
and to approximate $\pi(f)$ by the time average of this Markov chain.
There are infinitely many ways to construct such a discrete time process.
The general idea is to use an approximate time discretization of a
time-continuous process known to be ergodic with respect to $\pi$.
Generally, we cannot expect this discrete time process to be ergodic
with respect to $\pi$.
Thus one can use a Metropolis-Hastings acceptance-rejection step
that ensures the detailed balance, and thus makes the chain reversible
and ergodic with respect to $\pi$.
In the case of the Euler-Maruyama discretization of the overdamped Langevin dynamics,
\begin{align}
  \label{eq:ovLangevin_rev}
  dX_t=\nabla\log\pi(X)+\sqrt{2}dW_t,
\end{align}
with $(W_t)_{t\geq 0}$ a standard Brownian motion in $\R^d$,
the method is called Metropolis Adjusted Langevin Algorithm
(MALA, \cite{roberts1996}).

Yet, it has been noticed in several contexts that departing from the reversibility
can improve the performances of MCMC methods.
This article aims to propose a generalization of the standard MALA
that can be able construct nonreversible Markov chains that can outperform
classical MALA.

\subsection{Nonreversible dynamics}

On the continous time setting, analysis have been carried out to compare
the convergence properties of some time-continuous dynamics that are ergodic with
respect to $\pi$
\cite{hwang1993,hwang2005,LelievreNierPavliotis13,Wu2014,ReyBellet15,ReyBellet15_2,
Duncan2016},
based on two kinds of optimality criterion:
the speed of convergence toward equilibrium, measured in terms of
spectral gap in $L^2(\pi)$, and asymptotic variance for the time averages.
Obviously from a computational point of view
an increase of the spectral gap enables to reduce the burn-in,
and a reduction of the asymptotic variance leads to a decrease of
the computational complexity of the corresponding MCMC method.
%
These analysis compare, for different vector fields $\gamma$,
the overdamped Langevin dynamics given by,
\begin{align}
  \label{eq:ovLangevin_nrev}
  dX_t=\nabla\log\pi(X_t)dt+\gamma(X_t)dt+\sqrt{2}dW_t.
\end{align}
Under the condition of non explosion and that the vector field $\gamma$ is
taken such that $\nabla\cdot(\gamma \pi)=0$, this dynamic is ergodic with respect
to $\pi$.
Such vector fields can be constructed easily: for any skew-symmetric matrix $J$,
the vector field $\gamma$ defined by $\gamma(x)= J\nabla \log \pi(x)$ satisfies
this divergence-free equation.
Moreover, under this hypothesis, the following equation,
\begin{align}
  \label{eq:ovLangevin_hamil}
  dX_t=\gamma(X_t)dt,
\end{align}
conserves the energy $U$, which justifies the Hamiltonian denomination
of the term $\gamma$.
Moreover, this dynamic is time reversible if and only if $\gamma= 0$,
and in this case, the detailed balance is satisfied
(that is to say that the generator of the diffusion \eqref{eq:ovLangevin_nrev}
is self-adjoint in $L^2(\pi)$).
It is well known that among all vector fields $\gamma$ satisfying
the non explosion condition and such that $\nabla\cdot(\gamma \pi)=0$,
the dynamics given by \eqref{eq:ovLangevin_rev} in
the reversible case ($\gamma=0$)
has the worse rate of convergence in terms of spectral gap in $L^2(\pi)$
\cite{hwang1993,hwang2005,LelievreNierPavliotis13}.
Recent work has been done to construct divergence free (with respect to $\pi$)
perturbations of the drift
that achieve optimal convergence properties in the Gaussian case
\cite{LelievreNierPavliotis13,Wu2014}.
Recent works also show that breaking the non reversibility with such divergence free
perturbations on the drift also leads to improvement on the asymptotic variance.
It is shown in \cite{ReyBellet15} that the asymptotic variance decreases under
the addition of the irreversible drift.
Moreover, is has been shown in \cite{ReyBellet15_2} that the asymptotic
variance is monotonically decreasing with respect to the growth of the drift,
and the limiting behavior for infinitely strong drifts is characterized.
More recently, in \cite{Duncan2016} the authors investigate the dependence
of the asymptotic variance on the strength of the nonreversible perturbation.
\vspace{2mm}

On the discrete time setting, classical methods that depart from reversible sampling
consist in hybrid (Hamiltonian) MCMC \cite{Duane87,LelievreRoussetStoltz10}
and generalized hybrid MCMC methods \cite{Kennedy2001456}.
In the former method, the drift direction is chosen isotropically at each time step
and long time Hamiltonian integration is then carried out in this direction.
The latter can be seen as a generalization of the former that brings some inertia
in the direction of the Hamiltonian dynamics.
Another class of nonreversible samplers is composed by lifting methods.
They are designed in the discrete state space case, to construct a Markov chain
that satisfies some skew detailed balance
\cite{diaconis2000,Chen:1999:LMC:301250.301315,Hukushima12, Vucelja}.
They consist in increasing the state space to take into account
a privileged drift direction that is explored more efficiently.
More recently, Bierkens proposed an extension of the classical Metropolis-Hastings
algorithm to generate unbiased nonreversible Markov chain \cite{Bierkens2015}.
This is achieved by modifying the acceptance probability to depart from detailed
balance.
Eventually, a recent and quite different approach has been proposed in
\cite{bierkens2016zig} in the big data settings to circumvent the poor scalability
of standard MCMC methods.
The authors construct a continuous time piecewise deterministic Markov process.
It is a constant velocity model where the velocity direction switches at random
times with a rate depending on the target distribution.

\subsection{Outline}

We propose in this article a bias-free algorithm similar to MALA
that aims to exploit the asymptotic variance reduction of the nonreversible
time-continuous process.
The idea is to construct a Markov chain with invariant measure $\pi$,
by discretizing an equation of the form \eqref{eq:ovLangevin_nrev},
instead of equation \eqref{eq:ovLangevin_rev}, enhanced with an acceptance-rejection
step.
The main difficulty consists in unbiasing the unadjusted chain.
Indeed, it is not worth considering the use of a standard Metropolis-Hastings
acceptance probability since it is designed to
impose detailed balance with respect to the target distribution $\pi$,
and thus to define a reversible Markov chain with respect to $\pi$.
It would lead to a poor average acceptance ratio.
An elegant way would be to construct an adequate acceptance probability
with respect to this proposal, to ensure a high average acceptance ratio.
In the setting of \cite{Bierkens2015}, it would consist in finding a good
vorticity kernel.
Yet, we are not able to exactly do this.
Instead, we propose a class of lifted algorithms that rely on these unadjusted chains.
More precisely the first algorithm is a generalized Metropolis-Hastings algorithm
in an enhanced state space, and the second one can be seen as the analogous of
the generalized hybrid Monte Carlo method for the overdamped Langevin equation.
\vspace{2mm}

In section \ref{sec:GMALA} we present the first algorithm (generalized MALA).
We discuss in \ref{sec:GMALA_Q} how its performances are closely related to
the choice of the transition kernel of the unadjusted chain.
In section \ref{sec:GMALA_conv} we prove geometric convergence of the Markov
chain constructed with this algorithm under some hypotheses,
that ensure the existence of a central limit theorem.
Then, in section \ref{sec:GHMALA} we propose a modification of this algorithm
(the generalized hybrid MALA).
In section \ref{sec:num} we present numerical comparisons of these algorithms
with respect to classical MALA, which is followed by concluding remarks.
\vspace{2mm}

\section{Generalized MALA}
\label{sec:GMALA}

In this section, we construct a nonreversible Markov chain, ergodic with respect
to a target distribution $\pi$ known up to a normalizing constant.
The algorithm is similar to MALA in the sense that it constructs a Markov chain
from the discretization of an overdamped Langevin dynamic, augmented with an
acceptance-rejection step that makes it ergodic with respect to $\pi$.
The difference is that we construct a Markov chain on the discretization of
a nonreversible Langevin equation
to try to benefit from the smaller asymptotic variance of this kind of Markov
processes, than the reversible ones.
The main issue is then to choose a right acceptance probability that preserves the
good ergodic properties of the underlying Markov process.
\vspace{2mm}

To state our algorithm, we slightly modify Equation \eqref{eq:ovLangevin_nrev}.
For $\xi\in\R$, we consider the diffusions,
\begin{align}
  \label{eq:ovLangevin_nrev_xi}
  dX_t=\nabla\log\pi(X_t)dt+\xi\gamma(X_t)dt+\sqrt{2}dW_t,
\end{align}
with divergence-free condition $\nabla\cdot(\gamma \pi)=0$.
This way, $\xi$ specifies the direction and the intensity of the nonreversibility.
We denote now by $Q^{\xi}$ a proposal kernel that
correspond to some discretization of the diffusions \eqref{eq:ovLangevin_nrev_xi}
with parameter $\xi$.
We propose the following algorithm that we call Generalized MALA (GMALA),
\begin{algo}[Generalized MALA]
  \label{algo:GMALA}
  Let $ h>0$, $(x_0,\xi_0)\in\R^d\times\R$ be an initial point and an initial direction.
  Iterate on $n\geq 0$.
  \begin{enumerate}
    \item Sample $y^{n+1}$ according to $Q^{\xi^n}(x^n,dy)$.
    \item Accept the move with probability
    \begin{align}
      \label{eq:GMALA_ratio}
      A^{\xi^n}(x^n,y^{n+1})
      =1\wedge \dfrac{\pi(y^{n+1})Q^{-\xi^n}(y^{n+1},x^n)}{\pi(x^n)Q^{\xi^n}(x^n,y^{n+1})}.
    \end{align}
    and set $(x^{n+1},\xi^{n+1})=(y^{n+1},\xi^n)$;
    otherwise set $(x^{n+1},\xi^{n+1})=(x^n,-\xi^n)$.
  \end{enumerate}
\end{algo}
The important part of this algorithm is that the direction $\xi^n$ of the
Hamiltoninan exploration must be inverted at each rejection to ensure its
unbiasedness.
A good choice of $Q^{\xi}$ is given in the next section.
\vspace{2mm}

Unbiasedness is obvious since this algorithm is actually built as a Generalized
Metropolis-Hastings algorithm on the increased state space $E=\R^d\times \{-\xi_0,\xi_0\}$.
To simplify notations, we denote by $x_\xi$ all element $(x,\xi)\in E$.
We set $S$ the involutive transformation defined for all element
$x_\xi \in E$ by $S(x_\xi)=x_{-\xi}$.
We extend the definition of $\pi$ on $E$ by
$\pi(x_\xi)=\frac{\pi(x)}{2}$ for $x_\xi\in E$.
Obviously $\pi$ is unchanged by $S$.
Then, the algorithm constructs a Markov chain with transition kernel density $P$
given by,
\begin{align*}
  P(x_{\xi},y_{\eta})=Q^{\xi}(x,y)A^{\xi}(x,y)\mathds{1}_\xi(\eta)+\delta_{S(x_\xi)}(y_\eta)
  \left(1-\int_{\R^d}Q^{\xi}(x,z)A^{\xi}(x,z)dz \right),
\end{align*}
where $\mathds{1}_{xi}$ denotes the characteristic function and $\delta_{S(x_\xi)}$
a Dirac delta function,
that satisfies the following skew detailed balance,
\begin{align}
  \forall x_{\xi},y_{\eta}\in E,\quad
  \pi(x_{\xi})P(x_{\xi},y_{\eta})=\pi(y_{\eta})P(S(y_{\eta}),S(x_{\xi})).
\end{align}
\vspace{2mm}

Heuristically, we hope that the discretization of the time-continuous process
\eqref{eq:ovLangevin_nrev} specified by $Q^\xi$ benefits from the same good
behavior in terms of asymptotic variance reduction.
Since the Markov chain is constructed as parts of the discretization of the time-continuous
dynamics between the rejections, then the closer to one is the average
acceptance probability, the longer are these parts, and the more we can hope to
benefit from this good behavior.
Thus, to give a hint about the relevance of this algorithm, we compute in section
\ref{sec:GMALA_Q} these average acceptance probabilities, and we show
that they are of the same order of those for MALA for some well-chosen discretization.
Moreover, we numerically show in section \ref{sec:num} that it can outperform
MALA by several order of magnitude in terms of asymptotic variance.
\vspace{2mm}

Yet, before showing these results, we present a heuristic justification
about this algorithm.
In the reversible case, the time-continous equation satisfies the detailed balance
with respect to $\pi$,
\begin{align*}
  \forall x,y\in\R^d,\forall  h>0,\quad
  \pi(x)P_{ h}(x,y)=\pi(y)P_{ h}(y,x),
\end{align*}
where $P_{ h}(x,y)$ denotes the density of the transition kernel to
go from $x$ to $y$ after a time $ h$ for the dynamics \eqref{eq:ovLangevin_rev}.
Moreover, MALA imposes that the Markov chain also satisfies the detailed balance
with respect to $\pi$.
In the nonreversible case, the time continuous process \eqref{eq:ovLangevin_nrev_xi}
satisfies a skew detailed balance as stated by the following lemma,
\begin{lem}
  \label{lem:skewed-db-continuous}
  For all $x$, $y\in\R^d$, for all $\xi\in\R$ and for all $ h>0$,
  the following relation holds,
  \begin{align}
    \label{eq:skewed-db-continuous}
    \pi(x)P_{ h}^\xi(x,y)=\pi(y)P_{ h}^{-\xi}(y,x),
  \end{align}
  where $P_{ h}^\xi(x,dy)$ is the transition probability measure of
  the $ h$-skeleton of the process $(X^{\xi}_t)_{t\geq 0}$, solution of
  Equation \eqref{eq:ovLangevin_nrev_xi}.
\end{lem}

The analogous of the reversible case would be to construct a $\pi$-invariant Markov
chain that satisfies the same kind of skew detailed balance.
Because the skew detailed balance \eqref{eq:skewed-db-continuous} can be seen
as a detailed balance on the enhanced state space $E$ up to the transformation $S$,
\begin{align}
  \label{eq:skewed-db-continuous_gen}
  \forall x,y\in\R^d,\forall \xi\in\R, \forall  h>0\quad
  \pi(x)P_{ h}(x_{\xi},y_{\xi})=\pi(y)P_{ h}(S(y_{\xi}),S(x_{\xi})),
\end{align}
where $P_{ h}(x_{\xi},y_{\xi})=P^\xi_{ h}(x,y)$,
the Generalized Metropolis-Hastings method given by algorithm \ref{algo:GMALA}
is actually the classical way to construct such a Markov chain.
Nevertheless, the main difference between the time-continous and the discrete time
dynamics is the fact that the latter requires some direction switching of the
nonreversible component of the dynamics.
As stated previously, the idea of lifting the state space to construct a
nonreversible chain has been used in the discrete state space setting.
Yet, the idea is quite different.
With classical lifting methods, the goal is to switch between several
directions with well-chosen probabilities, to quickly explore the state space
in all of these directions.
In our case, we do not aim to switch between several nonreversible directions.
Yet, we are forced to do so at each rejection, and we have no choice but to
reverse the current nonreversible directions.
That is to say that we control neither the probability of switching nor the direction.

\begin{proof}[Proof of Lemma \ref{lem:skewed-db-continuous}]
We denote by $\mathcal{L}^\xi$ the generator of diffusion \eqref{eq:ovLangevin_nrev_xi}.
One can show that,
\begin{align*}
  \mathcal{L}^\xi=\mathcal{S}+\xi\mathcal{A},
\end{align*}
where $\mathcal{S}=\nabla\log\pi(x)\nabla\cdot$, and $\mathcal{A}=\gamma\cdot\nabla$.
Then, one can show that $\mathcal{S}$ and $\mathcal{A}$ are respectively the
symmetric and the antisymmetric parts of $\mathcal{L}^1$, with respect to $L^2(\pi)$.
Moreover, one can show that for all $\xi\in\R$, $\mathcal{L}^\xi$ and
$\mathcal{L}^{-\xi}$ are adjoint from one another.
This amounts to show that $D((\mathcal{L}^{-\xi})^*)\subset D(\mathcal{L}^{\xi})$,
which follows from the injectivity of the operator $(\mathcal{L}^{-\xi})^*-iI$,
and the surjectivity of the operator $\mathcal{L}^\xi-iI$.
It follows that the semigroups $e^{\mathcal{L}^\xi t}$ and $e^{\mathcal{L}^{-\xi} t}$
are adjoint from one another (Corollary 10.6 \cite{Pazy83}).
Then, for all bounded functions $f$ and $g$,
\begin{align*}
  \displaystyle\int_{(\R^d)^2}f(x)g(y)P_{ h}^\xi(x,dy)\pi(dx)
  &=\displaystyle\int_{\R^d}f(x)e^{\mathcal{L}^\xi  h}g(x)\pi(dx),\\
  &=\displaystyle\int_{\R^d}e^{(\mathcal{L}^\xi)^*  h}f(x)g(x)\pi(dx),\\
  &=\displaystyle\int_{\R^d}e^{\mathcal{L}^{-\xi}  h}f(x)g(x)\pi(dx),\\
  &=\displaystyle\int_{(\R^d)^2}f(y)g(x)P_{ h}^{-\xi}(x,dy)\pi(dx).
\end{align*}
\end{proof}
\vspace{2mm}

\subsection{Choice of the proposition kernel}
\label{sec:GMALA_Q}

The method presented above can be tuned with the choice of the proposition kernel $Q^\xi$.
It should be chosen such that it approximates the law of the transition probability
of the $ h$-skeleton of the process $(X_t^\xi)_{t\geq 0}$ solution of equation
\eqref{eq:ovLangevin_nrev_xi}.
The basic idea is to define $Q^\xi$ as the density of the law of an approximate
discretization of equation \eqref{eq:ovLangevin_nrev_xi}.
Doing so, we can hope that the skew-detailed balance for the unadjusted chain would
be almost satisfied (since it is satisfied for the time-continous dynamics),
that is to say we can hope to benefit from an acceptance ratio close to one.
Moreover, we recall that the choice of the proposal kernel is the only degree of
freedom that enables to tune the rate of the direction switching.
Formally, the closer the proposal kernel is from the transition kernel of the
$ h$-skeleton of the continuous process, the higher the acceptance ratio is
(and thus the less frequent is the direction switching).
\vspace{2mm}

The basic idea would be to propose according to the Maruyama-Euler approximation
of equation \eqref{eq:ovLangevin_nrev_xi}, as it is done for MALA.
We denote by $Q_1^{\xi}$ this kernel.
That is to say $Q_1^{\xi}(x,dy)$ is given by the law of $y$, solution of
\begin{align}
  \label{eq:Q1_eq}
  y=x- h\nabla U(x)- h\xi\gamma(x) + \sqrt{2 h} \chi,
\end{align}
where $\chi$ a standard normal deviate.
Then, $Q_1^\xi$ is given by
\begin{align}
  \label{eq:Q1_density}
  Q_1^\xi(x,dy)=
  \dfrac{1}{(4\pi h)^{d/2}}
  \exp\left(\frac{1}{4 h}\norms{y-(x- h\nabla U(x)- h\xi\gamma(x))}^2\right)dy.
\end{align}
Sadly, even though the simplicity of this proposal is appealing, this proposition
kernel leads to an average rejection rate of order $h$ when $\xi\neq 0$ and $\gamma$ is
non-linear as stated in Proposition \ref{prop:reject_ratio_fp}.
It is significantly worse than MALA that enjoys an average rejection rate of order
$ h^{3/2}$ (see \cite{MR2583309}).
As shown later by the numerical simulations, this bad rejection rate is not
a pure theoretical problem: it forbids to use large discretization steps $ h$,
and thus the method only generates highly correlated samples.
\vspace{2mm}

To overcome this problem, we propose to implicit the resolution of the nonreversible
term in equation \eqref{eq:ovLangevin_nrev_xi} with a centered point discretization.
More precisely, we propose a move $y_\xi$ from $x_\xi$, such that
$y$ would be distributed according to the solution of
\begin{align}
  \label{eq:Q2_eq}
  \Phi_x^{h\xi}(y)=x- h\nabla U(x) + \sqrt{2 h} \chi,
\end{align}
where $\chi$ a standard normal deviate, and $\Phi_x^{h\xi}$ is the function defined by,
\begin{align}
  \label{eq:phi_x^hxi}
  \Phi_x^{ h\xi}:\quad \R^d  \to  \R^d,\quad
  y\mapsto y+ h\xi\gamma\left(\dfrac{x+y}{2}\right).
\end{align}
Existence of such a proposal kernel, denoted by $Q_2^\xi$, is ensured
under the hypothesis that $U$ is twice differentiable with
uniformly bounded second derivative.
\begin{assumption}
  \label{as:U_Lipschitz}
    We suppose that $U$ is twice differentiable with uniformly bounded second
    derivative.
\end{assumption}

\begin{prop}
  \label{prop:existence_y}
  Under Assumption \ref{as:U_Lipschitz}, there exists $h_0>0$, such that
  for all $h<h_0$, for all $x_\xi\in E$, for all $\chi\in\R^d$,
  there exists a unique $y\in\R^d$, solution of equation \eqref{eq:Q2_eq},
  and if $\chi$ denotes a standard normal deviate, the function $Y_\xi$ defined
  almost surely as the solution of equation \eqref{eq:Q2_eq} is a well-defined
  random variable, with law $Q_2^\xi(x,dy)$, given by,
  \begin{align}
    \label{eq:Q2_density}
    Q_2^\xi(x,dy)=
    \dfrac{1}{(4\pi h)^{d/2}}
    |\jac{\phi^{\xi h}_x(y)}|
    \exp\left(\frac{1}{4 h}\norms{\phi^{ h\xi}_x(y)-(x- h\nabla\log\pi(x))}^2\right)dy.
  \end{align}
  Moreover
  \begin{align*}
    x-y=\sqrt{2h} \chi +O\left(h\norm{\nabla U(x)} + h^{3/2}\norm{\chi} \right),
  \end{align*}
  and there exists $C>0$, independent of $x$, $\chi$, and $h$, such that,
  \begin{align*}
    \norm{\nabla U(y)}\leq (1+Ch)\norm{\nabla U(x)} + C\sqrt{2h} \norm{\chi}.
  \end{align*}
\end{prop}
\begin{proof}[Proof of Proposition \ref{prop:existence_y}]
  The first point is a corollary of Lemma \ref{lem:diffeo} below.
  The second point can be proven by making use of the fact that $\nabla U$ is Lipschitz.
\end{proof}
The following lemma is used to prove Proposition \ref{prop:existence_y}.
We can note that the $h_0$ given by this lemma depends on the
Lipschitz coefficient of $\nabla U$.
\begin{lem}
  \label{lem:diffeo}
  Under Assumption \ref{as:U_Lipschitz}, there exists $h_0>0$, such that
  for all $h<h_0$, for all $x_\xi\in E$,
  the function $\Phi_x^{h\xi}:\R^d\to\R^d$, defined by \eqref{eq:phi_x^hxi}
  is a $C^1$-diffeomorphism.
\end{lem}
\begin{proof}[Proof of Lemma \ref{lem:diffeo}]
  The proof that $\Phi_x^{h\xi}$ is bijective relies on Picard fixed point theorem.
  For any fixed $z\in\R^d$, we define $\Psi$ on $\R^d$ by
  $\Psi(y)=z-h\xi J \nabla U \left( \dfrac{x+y}{2} \right)$.
  Then, since $\nabla U$ is supposed to be Lipschitz, then for small enough $h$
  this application is a contraction mapping from $\R^d$ to $\R^d$.
  The fact that $\Phi_x^{h\xi}$ is everywhere differentiable is clear, and the fact that
  its inverse is everywhere differentiable as well comes from the fact that
  the determinant of the Jacobian of $\Phi_x^{h\xi}$ is strictly positive
  for sufficiently small $h$.
\end{proof}
\vspace{2mm}

Sampling from the this proposal kernel can be done by a fixed point method
when no analytical solution is available since the proof of Lemma \ref{lem:diffeo}
uses a Picard fixed point argument.
This transition kernel leads to a rejection rate of order $ h^{3/2}$
(see Proposition \ref{prop:reject_ratio_fp}),
which is en improvement from $Q_1^\xi$.
The main advantage of this kernel is that the computation of
$|\jac{\phi^{\xi h}_x(y)}|$ can be avoided in the case where $\gamma$
is defined by $\gamma(x)=J\nabla \log \pi(x)$, with
$J$ a skew symmetric matrix.
Indeed, only the ratio $|\jac{\phi^{\xi h}_x(y)}|/|\jac{\phi^{-\xi h}_y(x)}|$
is required to compute the acceptance probability $A(x,y)$,
and this ratio is equal to $1$ if $\gamma$ is of gradient type, as stated by the following lemma.
\begin{lem}
  \label{lem:jac_pm}
  For all $x,y\in\R^d$, and for all $ h>0$,
  \begin{align*}
    |\jac{\phi^{\xi h}_x(y)}|=|\jac{\phi^{-\xi h}_y(x)}|
  \end{align*}
\end{lem}
\begin{proof}[Proof of lemma \ref{lem:jac_pm}]
  This result uses the more general fact that for any skew-symmetric
  matrix $A$ and any symmetric matrix $S$, the matrices $Id+AS$ and
  $Id-AS$ have same determinant.
  This statement is equivalent to say that $\chi_{AS}=\chi_{-AS}$,
  where we denote by $\chi_{M}$ the characteristic polynomial of any square
  matrix $M$.
  This last statement is true since for any square matrices $A$ and $S$ (of the same size)
  $\chi_{AS}=\chi_{SA}$.
  Then using the transposition, $\chi_{AS}=\chi_{A^t S^t}$.
  Eventually using the fact that $A^t=-A$ and $S^t=S$, it comes
  $\chi_{AS}=\chi_{-A S}$.
  The result follows from the fact that the matrix $\jac{\phi^{\xi h}_x(y)}$
  is of the form $Id+\xi AS$.
\end{proof}

We denote by $\alpha_{ h,\xi}^1$ and $\alpha_{ h,\xi}^2$
respectively the acceptance probability for proposal kernels $Q_1^\xi$
and $Q_2^\xi$ (defined respectively by \eqref{eq:Q1_density} and \eqref{eq:Q2_density}).
The following proposition provides an upper bound on the moments of the rejection
probability.
\begin{prop}
  \label{prop:reject_ratio_fp}
  Suppose that $U$ is three times differentiable with bounded second and third
  derivatives.
  Then for all $l\geq 1$, there exists $C(l)>0$ and $ h_0>0$ such that
  for all positive $ h< h_0$, and for all $x\in\R^d$,
  \begin{align*}
    \esp{(1-\alpha_{ h,\xi}^1(x,Y^1_{h,\xi}))^{2l}}&\leq C(l)(1+\norm{\nabla U(x)}^{4l})h^{2l}\\
    \esp{(1-\alpha_{ h,\xi}^2(x,Y^2_{h,\xi}))^{2l}}&\leq C(l)(1+\norm{\nabla U(x)}^{4l})h^{3l}
  \end{align*}
\end{prop}
\begin{proof}[Proof of Proposition \ref{prop:reject_ratio_fp}]
  To deal with both results at once, we define
  for all $x\in\R^d$, for all $\xi\in\{-1,1\}$,
  and for all $\theta\in\{0,1\}$, the random variable $Y_{x_\xi,\theta}$
  that satisfies the following implicit equation almost surely,
  \begin{align*}
    Y_{x_\xi,\theta} = x - h \nabla U(x) -  h \xi J \left( \theta \nabla
    U\left(\frac{x+Y_{x_\xi,\theta}}{2}\right)+(1-\theta)\nabla U(x) \right)+\sqrt{2h}\chi,
  \end{align*}
  where $\chi$ is a standard normal deviate in $\R^d$.
  Well-posedness of $Y_{x_\xi,\theta}$ is given by Proposition \ref{prop:existence_y}.
  We set $R^\xi_\theta$ the associated proposal kernel.
  We get $R_0^\xi=Q_1^\xi$
  and $R_1^\xi=Q_2^\xi$.
  We define the Metropolis-Hastings ratio $r(x_\xi,y_\xi)$ for proposing $y_\xi$
  from $x_\xi$ with kernel $R_\theta^\xi$ by,
  \begin{align*}
    r(x_\xi,y_\xi)=\dfrac{\pi(y_{-\xi}) R_\theta(y_{-\xi},x_{-\xi})}
    {\pi(x_{\xi})R_\theta(x_{\xi},y_{\xi})},
  \end{align*}
  where we set $\chi\in\R^d$ such that,
  \begin{align*}
    y = x - h \nabla U(x) -  h \xi J \left( \theta \nabla
    U\left(\frac{x+y}{2}\right)+(1-\theta)\nabla U(x) \right)+\sqrt{2h}\chi,
  \end{align*}
  Then, a straightforward computation gives,
  \begin{align*}
    \log(r(x_\xi,y_\xi))
    =&U(x)-U(y)+\langle y-x,\nabla U(x) \rangle\\
    &+\dfrac{1}{2} \langle y-x,\nabla U(y)-\nabla U(x) \rangle \\
    &- \dfrac{\xi}{2}(1-\theta) \langle y-x, J (\nabla U(y)-\nabla U(x)) \rangle \\
    &+\dfrac{ h}{4}\left( \norm{\nabla U(x)}^2 - \norm{\nabla U(y)}^2 \right)\\
    &+\dfrac{ h}{2} \xi \theta \langle J\nabla U\left( \dfrac{x+y}{2} \right),
    \nabla U(x)+\nabla U(y)  \rangle \\
    &+ \dfrac{ h}{2} \xi^2 \theta (1-\theta)
    \langle J\nabla U\left( \dfrac{x+y}{2} \right),
    J(\nabla U(x)-\nabla U(y)) \rangle\\
    &+\dfrac{ h}{4} (1-\theta)^2 \xi^2
    \left( \norm{J\nabla U(x)}^2-\norm{J\nabla U(y)}^2 \right).
  \end{align*}
  A Taylor expansion of these terms, making use of Proposition \ref{prop:existence_y}
  and noticing that $\theta (1-\theta)=0$ for $\theta\in\{0,1\}$ leads to
  \begin{align*}
    \log(r(x_\xi,y_\xi))
    =&-\xi  h(1-\theta)\langle \chi,J D^2U(x)\cdot \chi \rangle\\
    &+O( h^{3/2}(\norm{\nabla U(x)}^2+\norm{\chi}^2+\norm{\chi}^3))\\
    =&O( h(1-\theta)\norm{\chi}^2
    + h^{3/2}(\norm{\nabla U(x)}^2+\norm{\chi}^2+\norm{\chi}^3)).
  \end{align*}
  Moreover,
  \begin{align*}
    (1-1\wedge r(x_\xi,y_\xi))^{2l}
    =&O(\log(r(x_\xi,y_\xi))^{2l})\\
    =&O( h^{2l}(1-\theta)^{2l}\norm{\chi}^{4l}
    + h^{3l}(\norm{\nabla U(x)}^{4l}+\norm{\chi}^{4l}+\norm{\chi}^{6l})),
  \end{align*}
  and thus there exists $C(l)>0$ such that,
  \begin{align*}
    \esp{(1-1\wedge r(x_\xi,Y_{x_\xi,\theta}))^{2l}}
    \leq C(l)(1+\norm{\nabla U(x)}^{4l}) h^{3l}+C(l) h^{2l}(1-\theta)^{2l}
  \end{align*}
\end{proof}

The previous method is quite efficient since no computation of the Hessian of
$\log \pi$ is required.
Nevertheless, global Lipschitzness of $\nabla U$ is required to justify the method.
We propose a last kernel, denoted by $Q^\xi_3$, that does not require this hypothesis
but still require $\gamma$ to be of gradient type and the computation of the
Hessian of $U$.
More precisely, $Q^\xi_3(x,dy)$ is the law of $y$, solution of
\begin{align}
  \label{eq:Q3_eq}
  \left(Id+\frac{ h\xi}{2} J \hess(U)(x)\right)(y-x)=- h (Id+\xi J)\nabla U(x)
  + \sqrt{2 h} \chi,
\end{align}
where $\chi$ a standard normal deviate.
Then, $Q_3^\xi$ is given by
\begin{align}
  \label{eq:Q3_density}
  Q_3^\xi(x,dy)=
  \dfrac{1}{(4\pi h)^{d/2}}
  \det{M^\xi(x)}
  \exp\left(\frac{1}{4 h}\norms{M^\xi(x)(y-x) + h (Id+\xi J)\nabla U(x)}^2\right)dy,
\end{align}
with $M^\xi(x)=Id+\dfrac{ h \xi}{2}J\hess{U}(x)$.
This transition kernel offers also a rejection rate $a^3_{ h,\xi}$
of order $ h^{3/2}$.
\begin{prop}
  \label{prop:reject_ratio_h}
  Suppose that $U$ is three times differentiable with bounded second and third
  derivatives.
  Then for all $l\geq 1$, there exists $C(l)>0$ and $ h_0>0$ such that
  for all positive $ h< h_0$, and for all $x\in\R^d$,
  \begin{align*}
    \esp{(1-\alpha_{ h,\xi}^3(x,Y^2_{h,\xi}))^{2l}}\leq C(l)(1+\norm{\nabla U(x)}^{4l})h^{3l}
  \end{align*}
\end{prop}
\begin{proof}[Proof of Proposition \ref{prop:reject_ratio_h}]
  The proof involves the same arguments as in Proposition \ref{prop:reject_ratio_fp}
  and is left to the reader.
\end{proof}
\vspace{2mm}

Even though Kernel $Q_2^\xi$ enables to circumvent the bad acceptance ratio of
the explicit proposal given by $Q_1^\xi$, it raises some difficulties.
First, it requires the potential $U$ to be globally Lipschitz, which
is quite restrictive.
To use the GMALA method in the non globally Lipschitz case, one can resort to
importance sampling (to get back to a globally Lipschitz setting),
or use a globally Lipschitz truncation of the potential to build the proposition kernel.
The same idea is used in MALTA \cite{roberts1996}.
This last trick might lead to high average rejection ratio in the truncated regions
of the potential.
Moreover, the computation of the proposed move requires the resolution of a
nonlinearly implicit equation, that can be solved by a fixed point method,
which is more costly than the Euler-Maruyama discretization used by MALA.
Specific methods of preconditioning should be considered to accelerate the
convergence of the fixed point.
\vspace{2mm}

It would be interesting to propose a proposition kernel based on an explicit
scheme, that would achieve a better acceptance ratio than $Q_1^\xi$,
with a lower computational cost than $Q_2^\xi$.
The question of decreasing the Metropolis-Hastings rejection rate has been
recently studied in \cite{fathi2015improving}.
The authors propose a proposition kernel constructed from an explicit scheme,
which is a correction (at order $ h^{3/2}$) of the standard Euler-Maruyama
proposal.
Nevertheless, in our case this approach would not enable us to construct a proposition kernel based on
an explicit scheme with a high average acceptance ratio.
Work has also been done in this direction with the Metropolis–Hastings algorithm
with delayed rejection, proposed in \cite{TierneyMira99}, and more recently
\cite{banterle2015accelerating}.
Yet, it is unclear whether this approach could be use efficiently in our case.

\subsection{Convergence of GMALA}
\label{sec:GMALA_conv}

In this section, we only treat the case in which GMALA is used with $Q_2^\xi$.
To obtain a central limit theorem for the Markov chain built with GMALA,
it is convenient to prove the geometric convergence in total variation norm
toward the target measure $\pi$.
Because we require $\nabla U$ to be globally Lipschitz to ensure
well-posedness of GMALA with transition kernel $Q_2^\xi$,
MALA would be likely to benefit from geometric convergence in this case.
Indeed, it is well-known that MALA can be geometrically ergodic
when the tails of $\pi$ are heavier than Gaussian, under suitable hypotheses
(see Theorem 4.1 \cite{roberts1996}).
For strictly lighter tails, we cannot hope for geometric convergence
(see Theorem 4.2 \cite{roberts1996}).
This part is devoted to showing that under some hypotheses, GMALA can benefit
from geometric convergence.
\vspace{2mm}

Classically for MALA, the convergence follows on from the aperiodicity and the irreducibility
of the chain, since by construction the chain is positive with invariant measure $\pi$.
Under these conditions, the geometric convergence can be proven by exhibiting a
suitable Foster-Lyapunov function $V$ such that,
\begin{align}
  \label{eq:drift}
  \lim_{\abs{x}\to\infty} \dfrac{PV(x)}{V(x)}<1.
\end{align}
In our case, the situation is slightly different.
The Markov chain built with GMALA is still aperiodic and phi-irreducible.
This is a simple consequence of the surjectivity of the application
$\Phi_x^{h\xi}$ in Lemma \ref{lem:diffeo}.
To be able to prove a drift condition such as \eqref{eq:drift},
it is usually required to be able to show that the acceptance probability
for a proposed move starting from $x_\xi$ does not vanish in expectation
for large $x$, which is not always true in our setting.
More precisely, we can only say that the maximum of the two average acceptance
probabilities for the proposed moves starting from $x_\xi$ and $x_{-\xi}$
does not vanishes for large $x$.
Then one strategy could be to choose a Foster-Lyapunov function $V$ that decreases
in expectation when $\xi$ is changed to $-\xi$ after a rejection.
An other strategy, that we present in the following and that seems to be more
natural and more easily generalizable to potential $U$ that does
not satisfy the specific hypotheses we use in this section.
We show that the odd and the even subsequences of the Markov chain
converge geometrically quickly to the target measure $\pi$ by showing a drift condition
on $P^2$: the transition probability kernel of the two-steps Markov chain
defined by
\begin{align*}
  P^2(x,dy)=\int_E P(x,dz)P(z,dy),\quad\forall x\in E,
\end{align*}
under the following assumption.
\begin{assumption}
\label{as:control_U}
  We suppose that $U$ is three times differentiable, and that,
  \begin{enumerate}
    \item eigenvalues of $D^2 U(x)$ are uniformly upper bounded and lower
    bounded away from $0$, for $x$ outside of a ball centered in $0$,
    \item the product $\norm{D^3 U(x)}\norm{\nabla U(x)}$ is bounded
    for $x$ outside of a ball centered in $0$.
  \end{enumerate}
\end{assumption}

\begin{prop}
  \label{prop:drift2}
  Suppose assumption \ref{as:control_U}.
  There exists $s>0$ and $h_0>0$ such that for all $h\leq h_0$, for $\xi\in\{-1,1\}$,
  \begin{align}
    \label{eq:drift2}
    \lim_{\norm{x}\to + \infty} \dfrac{P^2V_s(x_\xi)}{V_s(x_\xi)}=0,
  \end{align}
  where the Foster-Lyapunov function $V$ is defined by $ V_s(x_\xi)=\exp( sU(x) )  $.
\end{prop}
\vspace{2mm}

\begin{rem}
  In order to prove the drift condition \eqref{eq:drift}, one can use the Foster-Lyapunov
  $\widetilde{V}$ defined by,
  \begin{align*}
    \widetilde{V_s}(x_\xi)=\exp\left( sU(x)+s\dfrac{\xi h^2 \langle \nabla U(x), D^2 U(x) J\nabla U(x) \rangle}{\abs{\xi h^2 \langle \nabla U(x), D^2 U(x) J\nabla U(x) \rangle }} \right),
  \end{align*}
  for $s$ small enough.
  Yet, this proof is left to the reader, but uses the arguments developed in the
  following.
\end{rem}

Assumption \ref{as:control_U} is not meant to be sharp.
We are not striving for optimality here, but instead we aim to propose a simple
criterion which is likely to be satisfied in smooth cases.

The first step to prove equation \eqref{eq:drift2}, is to show that the proposed
move decreases the Lyapunov function $V_s$ in expectation.
\begin{lem}
  \label{lem:contraction_V}
  Under assumption \ref{as:U_Lipschitz},
  there exists $s_0>0$ such that for all $s<s_0$,
  there exists $C_1,C_2>0$ such that for all $h<h_0$ given by Proposition \ref{prop:existence_y},
  and for all $x\in\R^d$ large enough (depending on $h$),
  \begin{align*}
    \esp{V_s(Y_{h,\xi})}\leq V_s(x)e^{-s\frac{3h(1-C_1 h)}{4}\norm{\nabla U(x)}^2+sC_2},
  \end{align*}
  where $Y_{h,\xi}$ is defined in Proposition \ref{prop:existence_y}.
  Thus, for $h$ small enough,
  \begin{align*}
    \esp{V_s(Y_{h,\xi})}\leq V_s(x)e^{-s\frac{h}{2}\norm{\nabla U(x)}^2+sC_2},
  \end{align*}
  and for $x$ large enough,
  \begin{align*}
    \esp{V_s(Y_{h,\xi})}\leq V_s(x)e^{-s\frac{h}{4}\norm{\nabla U(x)}^2}.
  \end{align*}
\end{lem}
\begin{proof}[Proof of Lemma \ref{lem:contraction_V}]
  For all $y\in\R^d$, we set $\chi\in\R^d$ such that,
  \begin{align*}
    y = x - h \nabla U(x) -  h \xi J \left( \theta \nabla
    U\left(\frac{x+y}{2}\right)+(1-\theta)\nabla U(x) \right)+\sqrt{2h}\chi,
  \end{align*}
  A Taylor expansion of $y$ in $h$ yields,
  \begin{align*}
    U(y)
    &=U(x)+\nabla U(x)\cdot(y-x)+O(\norm{y-x}^2)\\
    &=U(x)-h\norm{\nabla U(x)}^2 + \sqrt{2h}\nabla U(x)\cdot\chi
     +h\xi\nabla U(x)\cdot J \nabla U(\frac{x+y}{2})
     +O(\norm{x-y}^2).
  \end{align*}
  Noticing that
  \begin{align*}
    h\xi\nabla U(x)\cdot J \nabla U(\frac{x+y}{2})
    = h\xi\nabla U(x)\cdot J (\nabla U(\frac{x+y}{2})-\nabla U(x)),
  \end{align*}
  the Cauchy-Schwarz and the triangular inequality yield,
  \begin{align*}
    \abs{h\xi\nabla U(x)\cdot J \nabla U(\frac{x+y}{2})}
    \leq C h \norm{\nabla U(x)} \norm{x-y}.
  \end{align*}
  Thus, by Young inequality and Proposition \ref{prop:existence_y}, for $h\leq 1$,
  \begin{align*}
    U(y)\leq U(x)-\dfrac{3h(1-Ch)}{4}\norm{\nabla U(x)}^2
    +O(\norm{\chi}^2),
  \end{align*}
  The conclusion holds for small enough $s$ such that $e^{sO(\norm{G}^2)}$
  is integrable, where $G$ is a standard normal deviate in $\R^d$.
\end{proof}

Classically (\cite{roberts1996}), proofs of geometric convergence of MALA require to show that
the average acceptance ratio of the proposed move from $x$, does not vanishes when $x$ is large.
Namely that there exists $\epsilon>0$ such that
\begin{align*}
  I(x)=\left\{y:\alpha(x,y) \leq 0 \right\}
\end{align*}
asymptotically has $q$-measure $0$, where $\alpha$ is the acceptance probability.
Our case is sightly different since it is possible that the average acceptance ratio
of the proposed move from $x_\xi\in E$ vanishes when $\norm{x}\to+\infty$,
which leads to a rejection and to the switching $\xi\leftarrow -\xi$ with probability
close to one.
Yet the average acceptance probability of the next proposed move from $x_{-\xi}$
is close to one.
This behavior is described by the following Lemma.
\begin{lem}
  \label{lem:alpha_pm}
  Under Assumption \ref{as:control_U}, there exists $h_0>0$ such that for all $h<h_0$
  and for all $\epsilon>0$, there exists $C(h,\epsilon)>0$ such that
  for all $x_\xi\in E$ such that $\norm{x}\geq C(h,\epsilon)$,
  \begin{align*}
    \xi \langle \nabla U(x),D^2 U(x)\cdot J \nabla U(x)\rangle\geq 0
    \Longrightarrow
    \prob{\alpha^2_{h,\xi}(x_\xi,y_\xi)=1}\geq 1-\epsilon,
  \end{align*}
  where $y$ is defined by Equation \eqref{eq:Q2_eq}.
\end{lem}
\begin{proof}[Proof of Lemma \ref{lem:alpha_pm}]
  We recall that $\alpha^2_{h,\xi}$ is defined for all $x_\xi,y_\xi\in E$
  by $\alpha^2_{h,\xi}(x_\xi,y_\xi)=1\wedge e^{r(x_\xi,y_\xi)}$, with,
  \begin{align*}
    r(x_\xi,y_\xi)
    &=U(x)-U(y)
    +\langle\sqrt{2h} \chi,\nabla U(x)\rangle
    +\dfrac{1}{2}\langle \sqrt{2h} \chi,\nabla U(y)-\nabla U(x) \rangle\\
    &-h\norm{\nabla U(x)}^2
    -h \langle \nabla U(y)-\nabla U(x), \nabla U(x) \rangle
    -\dfrac{h}{4}\norm{\nabla U(x)-\nabla U(y)}^2.
  \end{align*}
  The proof is done by computing a Taylor expansion in $h$ of this quantity to evaluate
  its sign in the asymptotic case $\norm{x}\to+\infty$.
  We denote by $\cdot$ the matrix vector product, and by $:$ and $\vdots$ respectively
  the double and triple dot products.
  \begin{align*}
    &U(x)-U(y)+\langle\sqrt{2h} \chi,\nabla U(x)\rangle-h\norm{\nabla U(x)}^2\\
    &\qquad=h\xi \nabla U(x)\cdot J\left( \nabla U\left( \dfrac{x+y}{2} \right)-\nabla U(x) \right)
    -\dfrac{1}{2}D^2 U(x):(y-x)^2+O(\norms{D^3 U(x)\vdots (y-x)^3})\\
    &\qquad=\dfrac{h}{2}\xi \nabla U(x)\cdot J\left( D^2 U(x)\cdot(y-x)+O(D^3 U(x):(y-x)^2) \right)\\
    &\qquad-\dfrac{1}{2}D^2 U(x):(y-x)^2+O(\norms{D^3 U(x)\vdots (y-x)^3})\\
    &\dfrac{1}{2}\langle \sqrt{2h} \chi,\nabla U(y)-\nabla U(x) \rangle
    =\dfrac{1}{2}\langle \sqrt{2h} \chi,D^2 U(x)\cdot (y-x)+O(D^3 U(x):(y-x)^2) \rangle,\\
    &-h \langle \nabla U(y)-\nabla U(x), \nabla U(x) \rangle
    =-h \langle D^2 U(x)\cdot (y-x)+O(D^3 U(x):(y-x)^2), \nabla U(x) \rangle\\
    &-\dfrac{h}{4}\norm{\nabla U(x)-\nabla U(y)}^2=hO(\norms{x-y}^2).
  \end{align*}

  Collecting these terms, and using Proposition \ref{prop:existence_y} and Assumption
  \ref{as:control_U}, it comes for $h\leq 1$
  \begin{align*}
      r(x_\xi,y_\xi)=-\dfrac{h}{2}\langle \nabla U(x),D^2 U(x)\cdot(y-x) \rangle
      +O(h^3 \norm{\nabla U(x)}^2)
      +O(\norm{\chi}^2+\norm{\chi}^3),
  \end{align*}
  where the big $O$ are independent of $h$.
  and eventually,
  \begin{align*}
    r(x_\xi,y_\xi)
    &=
     \dfrac{h^2}{2}\left(
     \langle \nabla U(x),D^2 U(x)\cdot \nabla U(x)\rangle
     +\xi \langle \nabla U(x),D^2 U(x)\cdot J \nabla U(x)\rangle
     \right)\\
    &+O(h^3 \norm{\nabla U(x)}^2)
     +O(\norm{\chi}^2+\norm{\chi}^3),
  \end{align*}
  Then, using the definite positivity of $D^2 U(x)$ and the fact that
  $\norm{\nabla U(x)}$ is coercive, for all $h$ small enough and for all
  $\epsilon>0$,
  there exists $C(h,\epsilon)>0$ such that
  for all $x\in\R^d$ such that $\norm{x}\geq C(h,\epsilon)$,
  \begin{align*}
    \xi \langle \nabla U(x),D^2 U(x)\cdot J \nabla U(x)\rangle\geq 0
    \Longrightarrow
    \prob{r(x_\xi,y_\xi)\geq 0}\geq 1-\epsilon.
  \end{align*}
\end{proof}

Proposition \ref{prop:drift2} can now be obtained as a consequence of Lemmas \ref{lem:alpha_pm}
and \ref{lem:contraction_V}.
\begin{proof}[Proof of proposition \ref{prop:drift2}]
  Because of the acceptance-rejection step, we get for all $x_\xi\in E$,
  \begin{align*}
    P^2V(x)=
    &\int_{\R^d} q(x_\xi,y_\xi)\alpha^2_{h,\xi}(x_\xi,y_\xi)\int_{\R^d} q(y_\xi,z_\xi)\alpha^2_{h,\xi}(y_\xi,z_\xi)V(z_\xi) dz dy\\
    &+\int_{\R^d} q(x_\xi,y_\xi)\alpha^2_{h,\xi}(x_\xi,y_\xi) V(y_{-\xi})
    \left(1-\int_{\R^d} q(y_\xi,z_\xi)\alpha^2_{h,\xi}(y_\xi,z_\xi) dz\right) dy\\
    &+\left(1-\int_{\R^d} q(x_\xi,y_\xi)\alpha^2_{h,\xi}(x_\xi,y_\xi) dy \right)
    \left( \int_{\R^d} q(x_{-\xi},y_{-\xi})\alpha^2_{h,\xi}(x_{-\xi},y_{-\xi}) V(y_{-\xi}) dy\right)\\
    &+V(x_{-\xi})\left(1-\int_{\R^d} q(x_\xi,y_\xi)\alpha^2_{h,\xi}(x_\xi,y_\xi) dy \right)
  \left(1- \int_{\R^d} q(x_{-\xi},y_{-\xi})\alpha^2_{h,\xi}(x_{-\xi},y_{-\xi}) dy \right).
  \end{align*}
  Simply using $\alpha^2_{h,\xi}\leq 1$ leads to,
  \begin{align*}
    P^2V(x)\leq
    &\int_{\R^d} q(x_\xi,y_\xi)\int_{\R^d} q(y_\xi,z_\xi)V(z_\xi) dz dy\\
    &+\int_{\R^d} q(x_\xi,y_\xi) V(y_{-\xi})
    \left(1-\int_{\R^d} q(y_\xi,z_\xi)\alpha^2_{h,\xi}(y_\xi,z_\xi) dz\right) dy\\
    &+\left(1-\int_{\R^d} q(x_\xi,y_\xi)\alpha^2_{h,\xi}(x_\xi,y_\xi) dy \right)
    \left( \int_{\R^d} q(x_{-\xi},y_{-\xi})V(y_{-\xi}) dy\right)\\
    &+V(x_{-\xi})\left(1-\int_{\R^d} q(x_\xi,y_\xi)\alpha^2_{h,\xi}(x_\xi,y_\xi) dy \right)
    \left(1- \int_{\R^d} q(x_{-\xi},y_{-\xi})\alpha^2_{h,\xi}(x_{-\xi},y_{-\xi}) dy \right).
  \end{align*}
  And eventually Lemmas \ref{lem:alpha_pm} and \ref{lem:contraction_V} gives,
  \begin{align*}
    \dfrac{P^2V(x)}{V(x)}\leq\quad
    &e^{-s\frac{h}{2}\norm{\nabla U(x)}^2+2sC_2}
    +e^{-s\frac{h}{2}\norm{\nabla U(x)}^2+sC_2}\\
    &+(1-\esp{\alpha^2_{h,\xi}(x_\xi,Y_{h,\xi})})e^{-s\frac{h}{2}\norm{\nabla U(x)}^2+sC_2}\\
    &+(1-\esp{\alpha^2_{h,\xi}(x_\xi,Y_{h,\xi})})(1-\esp{\alpha^2_{h,\xi}(x_{-\xi},Y_{h,-\xi})}).
  \end{align*}
  And thus,
  \begin{align*}
    \lim_{\norm{x}\to + \infty} \dfrac{P^2V(x)}{V(x)}=0
  \end{align*}
\end{proof}

\section{Generalized hybrid MALA}
\label{sec:GHMALA}

We propose in this part a second algorithm.
It is based on a splitting method by solving by turns the reversible and
the non reversible parts of equation \eqref{eq:ovLangevin_nrev_xi}
with measure preserving schemes.
The idea is then to integrate the pure reversible equation
\begin{align}
  \label{eq:rev}
  dx=-\nabla U(x)dt+dW_t.
\end{align}
by using MALA,
and to integrate the pure Hamiltonian equation
\begin{align}
  \label{eq:nrev}
  dx=-\xi J\nabla U(x)dt.
\end{align}
by an hybrid Monte-Carlo method based on a suitable Hamiltonian integrator.
This method presents some theoretical advantages on Generalized MALA.
Before presenting them, we begin with explaining the algorithm.
\vspace{2mm}

We denote by $\Psi^\xi_{t}$ the flow of the Hamiltonian equation \eqref{eq:nrev}
on the time interval $[0,t]$, and by $\Phi^\xi_{ h}$ a numerical
integrator for \eqref{eq:nrev} on a time step $ h$.
We precise later necessary conditions on this integrator that ensure unbiasedness
of the algorithm.

\begin{algo}[\textbf{Generalized Hybrid MALA}]
  Let $x_0$ be an initial point.
  Set $\xi=\pm 1$.
  Let $ h>0$.
  Iterate on $n\geq 0$,
  \begin{enumerate}
    \item Integration of the reversible part \eqref{eq:rev}:\\
    MALA is used to sample $x_{n+1/2}$ from $x_n$, with time-step $ h$.
    \item Integration of the non reversible part \eqref{eq:nrev}:
    \begin{enumerate}
      \item Compute $\widetilde{x}_{n+1}=\Phi^\xi_{ h}(x_{n+1/2})$.
      \item Set $x_{n+1}=\widetilde{x}_{n+1}$ with probability
      \begin{align*}
        \beta_{ h,\xi}(x_{n+1/2}) = \min\left(1,\exp(U(x_{n+1/2})-U(\widetilde{x}_{n+1}))\right)
        \min\left(1,\exp(U(x_{n+1/2})-U(\Phi^\xi_{h}(x_{n+1/2})))\right).
      \end{align*}
      Otherwise set $x_{n+1}=x_{n+1/2}$ and $\xi\leftarrow-\xi$.
    \end{enumerate}
  \end{enumerate}
\end{algo}
Similarly to the Hybrid Monte-Carlo algorithm, the first step enables to
explore the state space across the iso-potential lines, whereas the second step
enables to explore it along the iso-potential lines.
\vspace{2mm}

To ensure unbiasedness of the second step, the integrator $\Phi^\xi_{ h}$
must satisfy the following properties,
\begin{align}
  \label{eq:cond_Phi1}
  \Phi^\xi_{ h}=(\Phi^{-\xi}_{ h})^{-1},\\
  \label{eq:cond_Phi2}
  \det{\jac{\Phi^\xi_{ h}}}=1.
\end{align}
These properties are classical for hybrid Monte-Carlo methods (see Chapter 2.
\cite{LelievreRoussetStoltz10}).
\begin{lem}
  \label{lem:unbiasedness_hybrid}
  Under conditions \eqref{eq:cond_Phi1} and \eqref{eq:cond_Phi2}, the second step
  leaves the measure $\pi$ invariant.
\end{lem}
\begin{proof}[Proof of Lemma \ref{lem:unbiasedness_hybrid}]
  The proof can be found in \cite{LelievreRoussetStoltz10}.
  It consists in seeing this step as a generalized Metropolis-Hastings step,
  with proposal $Q(x_\xi,y_{\eta})=\delta_{\Phi^\xi_{ h}(x)} \delta_{\xi}(\eta)$,
  and symmetric operator $S(x_\xi)=x_{-\xi}$.
  Then, it is enough to show that the Metropolis-Hastings ratio $r$ defined by
  the following Radon-Nikodym derivative
  \begin{align*}
    r(x_\xi,y_{\eta})=\dfrac{Q(S(y_\eta),S(dx_\xi)) \pi(dy_\eta)}
    {Q(x_\xi,dy_\eta)\pi(dx_\xi)},
  \end{align*}
  is equal to $\exp\left(\log \pi(y) -\log \pi(x) \right)$.
\end{proof}

For example, for all $x_\xi\in E$ the centered point integrator
$\Phi^\xi_{ h}(x)$ defined by the solution $y$ of the following equation
\begin{align}
  \label{eq:integrator_hybrid}
  y=x- h \xi J \nabla \log \pi\left(\dfrac{x+y}{2} \right),
\end{align}
satisfies these properties, under some assumptions that ensure well-posedness.

\begin{lem}
  \label{lem:well-posedness_hybrid}
  Under Assumption \ref{as:U_Lipschitz},
  there exists $ h_0>0$ such that for all positive $ h< h_0$,
  there exists a unique solution to equation \ref{eq:integrator_hybrid}, and
  the integrator $\Phi^\xi_{ h}$ is well-defined and satisfies equations
  \eqref{eq:cond_Phi1} and \eqref{eq:cond_Phi2}.
\end{lem}
\begin{proof}[Proof of Lemma \ref{lem:well-posedness_hybrid}]
  The proof follows on from Lemma \ref{lem:diffeo}.
\end{proof}
\vspace{2mm}

The main benefit of this algorithm with respect to GMALA, is the better
average acceptance ratio of the Hybrid step with the centered point integrator,
which is of order $O( h^3)$ instead of $O( h^{3/2})$ for GMALA,
which may enable to reduce the rate of switching directions.
\begin{lem}
  \label{lem:hybrid_alpha}
  Suppose Assumption \ref{as:control_U}.
  Then, for all $l\geq 1$, there exists $ h_0>0$ such that for all positive
  $ h< h_0$, and for all $x\in\R^d$,
  \begin{align*}
    (1-\beta_{ h,\xi}(x))^{2l}\leq C\norm{\nabla U(x)}^{2l}h^{6l}.
  \end{align*}
\end{lem}
\begin{proof}[Proof of Lemma \ref{lem:hybrid_alpha}]
  For $x\in\R^d$, we set $y=\Phi^\xi_{ h}(x)$.
  A Taylor expansion of $U(y)$ yields,
  \begin{align*}
    \abs{U(y)-U(x)} = O(h^3\norm{D^3 U(x)}\norm{\nabla U(x)}^3).
  \end{align*}
\end{proof}

The centered point integrator is an example of integrator for the Hamiltonian
dynamics, that can be used as soon as the potential $U$ is globally Lipschitz.
Actually, similarly to GMALA in the non globally Lispchitz case,
one can construct an approximate integrator $\Phi_h^\xi$ by using a truncation
of $\nabla U$ to ensure its global Lipschitzness.
Yet, this trick may lead to high rejection rates in the areas where $\nabla U$
is truncated.
Again, other strategies can be used like importance sampling or even a change of
variable in the integrand \eqref{eq:integral}.
The GHMALA algorithm is especially interesting when we are able to integrate
efficiently the Hamiltonian dynamics.
We show later on the numerical experimentations two examples of specific integrators
that enable to improve significantly the performance of GHMALA with respect to
GMALA with proposal kernel $Q_2^\xi$.
We can recall for example, the case of separated Hamiltonian dynamics that can
be integrated with explicit volume preserving schemes which are time-reversible
and symmetric (\cite{qin1993}).
\vspace{2mm}

The authors propose in \cite{Duncan2016} a similar splitted scheme
where the hybrid step is replaced by fourth-order Runge-Kutta method.
Even though this choice leads to a biased estimator, it enabled to get rid of the centered
point integrator, that may be more costly than the Runge-Kutta method
in terms of computation time.
\vspace{2mm}

\section{Numerical experimentations}
\label{sec:num}

\subsection{Anisotropic distribution}

In this section, we test our nonreversible MALA algorithm by computing
an observable $f$ with respect to a two dimensional anisotropic distribution.
More precisely we want to estimate $\esp{f(X)}$ with
$f((x_1,x_2))=x_1^2 1_{x_1>15}$, and $X\sim \pi(x)$ with
$\pi(x)\propto e^{-V(x)}$, where
\begin{align*}
  V((x_1,x_2))=\dfrac{x_1^2}{\sqrt{1+50 x_1^2}}+x_2^2.
\end{align*}
Such a distribution is more stretched out in the $x_1$ direction rather than
the $x_2$.
We expect our lifted algorithm to be more favorable than MALA when the
anisotropy is strong.
Indeed, MALA tends to perform a slower exploration of the state space in the
$x_1$ direction with respect to $x_2$ the direction.
The lifted algorithm is supposed to correct this problem since the Hamiltonian
dynamics should lead to a fast exploration of the iso-potential lines,
which are stretched in the direction $x_1$.
We choose as the skew-symmetric matrix $J$ defined by,
\begin{align}
  \label{eq:J_numerical}
  J=\alpha\begin{pmatrix}0 & 1\\-1 & 0\end{pmatrix},
\end{align}
with $\alpha$ a real parameter.

To compare MALA with GMALA and GHMALA, different optimal time steps parameters
should be used.
There is a tradeoff between achieving high average acceptance ratios (obtained with
small time steps) and small correlation between the successive samples
(obtained with large time-steps).
More precisely MALA usually gives its best results with a large $ h$ that
ensures a significant average rejection ratio.
On the contrary, GMALA and GHMALA require much smaller time steps to avoid the
regress that happens with the direction switching at each rejection.
Thus, MALA should be used with a higher time step than the two others.
Moreover, as GMALA with proposal kernel $Q_1^\xi$ leads to a worse average
acceptance ratio than $Q_2^\xi$, the former requires a smaller time-step than
the latter.
\begin{figure}[!ht]
    \centering
    \includegraphics[width=0.70\textwidth]{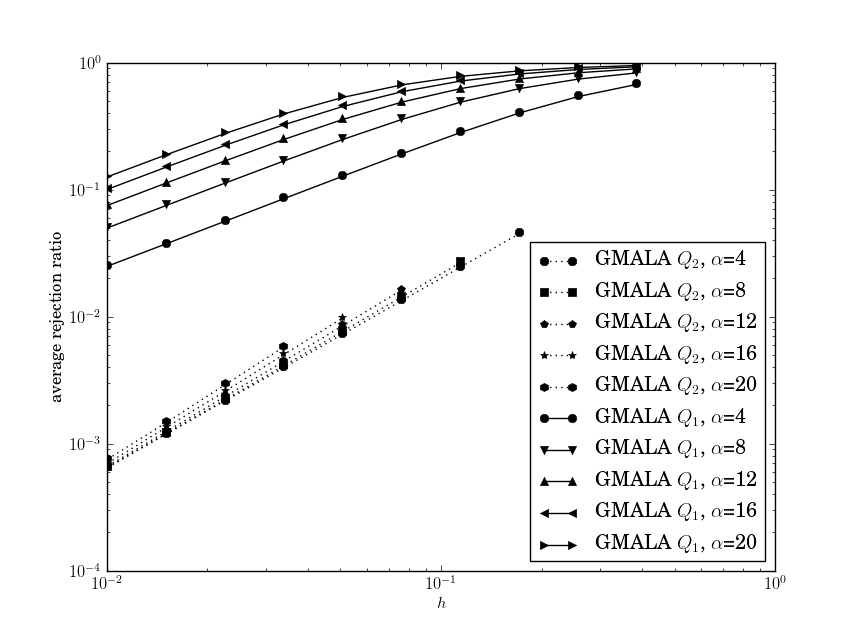}
    \caption{Comparison of average rejection ratio for GMALA with proposal kernels
    $Q_1^\xi$ and $Q_2^\xi$.}
    \label{fig:ratio_comp_Q1_Q2}
\end{figure}
Figure \ref{fig:ratio_comp_Q1_Q2} shows the average acceptance ratio
with respect to the time step $ h$ for GMALA with proposal kernels
$Q_1^\xi$ and $Q_2^\xi$.
We can verify that the average rejection ratio of GMALA with $Q_1^\xi$ is indeed
much worse than MALA and is of order $ h$.
In practice, this limitation leads to very correlated successive samples,
and thus bad asymptotic variances for the estimators based on such Markov chain.
In the following, we always consider GMALA with proposal kernel $Q_2^\xi$
instead of $Q_1^\xi$.
\vspace{2mm}

\begin{figure}[!ht]
    \centering
    \includegraphics[width=0.70\textwidth]{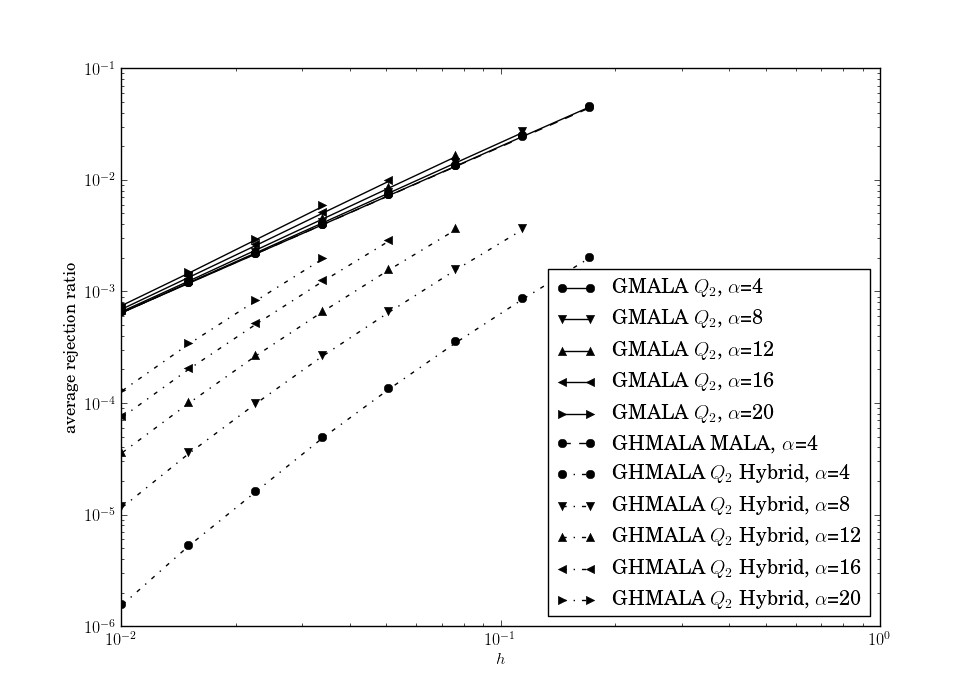}
    \caption{Comparison of average rejection ratio for MALA, GMALA (with proposal kernel $Q_2^\xi$) and GHMALA.}
    \label{fig:ratio_comp_GMALA_GHMALA}
\end{figure}
Figure \ref{fig:ratio_comp_GMALA_GHMALA} shows the
acceptance ratio with respect to the time step $ h$
for those three algorithms.
GHMALA is composed by two steps: the first one consists of MALA and the second one
consists of a Hybrid iteration.
Both steps are composed by an acceptanceection step.
The average acceptance rate of the first step does not depend on $\alpha$,
and is denoted by \textit{GHMALA MALA} in the legend.
Moreover, it is the same as a plain MALA.
The average acceptance rate of the second step is denoted by \textit{GHMALA Hybrid}.
We also observe that the rejection ratio for GMALA with $Q_2^\xi$ is close to MALA's and scales
as $ h^{3/2}$
This is due to the fact that the non reversibility contributes at order $ h^2$
in the expression of the average rejection ratio.
We can also verify the upper bound on the rejection probability of
the hybrid step, given by Lemma \ref{lem:hybrid_alpha}.
\vspace{2mm}

\begin{figure}[!ht]
    \centering
    \includegraphics[width=0.70\textwidth]{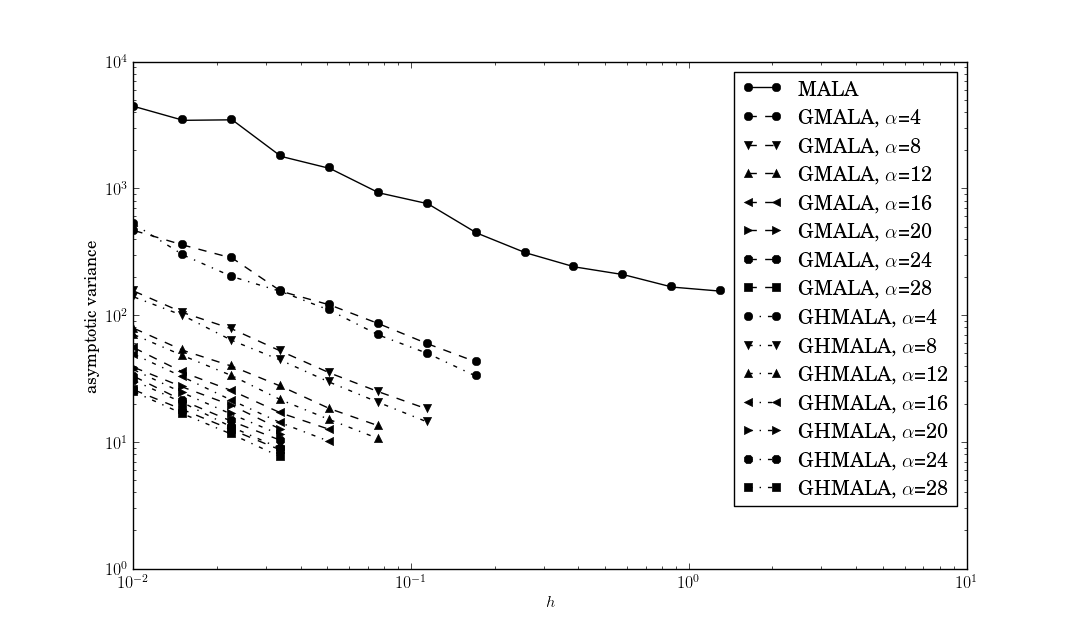}
    \caption{Variance comparison of MALA, GMALA ($Q_2^\xi$) and GHMALA on the anisotropic
    distribution}
    \label{fig:var_comp}
\end{figure}
We compare now the asymptotic variance of the estimators build with MALA, GMALA
and GHMALA.
To do so, we compute 1000 independent estimators composed by the time average of
$10^5$ samples.
We plot in Figure \ref{fig:var_comp} the empirical relative variance of these estimators,
We observe that GMALA performs much better than MALA by a factor 20.
We should precise that a reduction in the variance does not necessarily mean
a reduction in computing time since one iteration of GMALA is more costly
than MALA since it requires a fixed point iteration.


The question of choosing $\alpha$ is quite natural.
In the case of the time-continuous process, it is known that the decrease in
asymptotic variance is monotonic with the intensity $\alpha$, which suggests to
use the algorithms with large $\alpha$ \cite{ReyBellet15,Duncan2016,Hwang20153522}.
Yet, well-posedness of the proposal kernel $Q_2^\xi$ requires the product
$\alpha  h$ to be small enough to ensure convergence of the fixed point.
From a computational point of view, the cost of the method is proportional to
the number of Picard iterations, which scales like $-\log \rho$,
where $\rho$ denotes the contraction ratio (that scales as $\alpha  h$).
Then, the choice of parameters $\alpha$ and $ h$ should take into account
these two effects.

\subsection{Warped Gaussian distribution}

This example deals with a non quadratic potential.
Both GMALA and GHMALA can be adapted to this case.
The simple idea is to build a proposal kernel with a truncation of $\nabla U$,
to make it globally Lipschitz.
This is the same idea as the one used for MALTA (\cite{roberts1996}).
Moreover, it is also possible to choose a more efficient integrator
than the centered point integrator defined by Equation \eqref{eq:integrator_hybrid}.
\vspace{2mm}

To illustrate these two methods, we test now our algorithms in the case of a
two-dimensional warped Gaussian distribution.
This toy case has been introduced in \cite{haario2001} and used as a benchmark
for variance reduction methods based on nonreversible Langevin samplers in
\cite{Duncan2016}.
More precisely, we aim to estimate $\esp{f(X)}$ with the observable $f$
and $X$ distributed with $\pi$, defined by,
\begin{align*}
  f(x)=x_1^2+x_2^2,\quad
  \pi(x)\propto e^{-V(x)},\quad\text{with}\quad
  V(x)=\dfrac{x_1^2}{100}+\left(x_2+\dfrac{x_1^2}{20}-5 \right)^2.
\end{align*}
We define the skew symmetric matrix $J$ by \eqref{eq:J_numerical}.
It appeared in \cite{Duncan2016} that the nonreversible Langevin dynamics
enables to reduce the asymptotic variance by several orders of magnitude.
\vspace{2mm}

We propose to implement GMALA using a truncated drift to make it globally Lipschitz
to ensure the well-posedness of the method,
and to implement GHMALA using a specific integrator of the Hamiltonian dynamics.
We show that GHMALA performs better than GMALA and MALA in this casee.
More precisely, the integrator is defined as the centered point integrator, after
the symplectic change of variable $\psi$ defined for all $(x_1,x_2)\in\R^2$ by
\begin{align*}
  \psi(x_1,x_2)=(x_1,x_2+\dfrac{x_1^2}{20}-5).
\end{align*}
Typically, this integrator enables to solve this dynamics for larger time steps
than GMALA with proposal kernel $Q_2^\xi$, and thus reduces the asymptotic variance.
Figure \ref{fig:var_warp} displays the asymptotic variance for the estimators build
with these algorithms.
It is computed as the empirical variance of 2000 independent estimators constructed
as the time average of $10^5$ iterations of the algorithms.
\begin{figure}[!ht]
    \centering
    \includegraphics[width=0.70\textwidth]{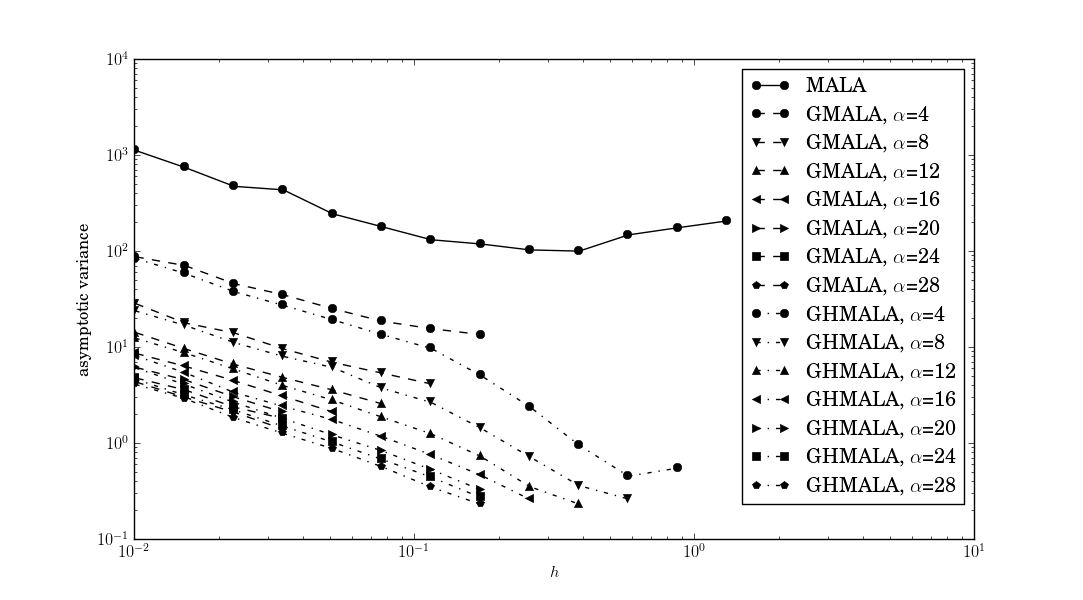}
    \caption{Variance comparison of MALA, GMALA and GHMALA on the warped Gaussian
    distribution}
    \label{fig:var_warp}
\end{figure}
We can observe that GMALA and GHMALA performs similarly for small time steps
$ h$.
Yet, for larger time-steps, it is not possible to define the proposal kernel
for GMALA, whereas it is still the case for GHMALA.
Eventually, we achieve a variance reduction of about a factor 500 with GHMALA and
60 with GMALA, compared with classical MALA.

\subsection{Quartic Gaussian distribution}

This toy case aims to present a particular case where GHMALA can be
used without implicit integrator, and in a non globally Lipschitz case,
which may enable to reduce the computational cost by avoiding the fixed point
iteration.
Again, we aim to estimate $\esp{f(X)}$ where $X$ is distributed with $\pi$,
and where,
\begin{align*}
  f(x)=x_1^2+x_2^2,\quad
  \pi(x)\propto e^{-V(x)},\quad\text{with}\quad
  V(x)=\dfrac{x_1^2}{100}+x_2^4.
\end{align*}
We define the skew-symmetric matrix $J$ by \eqref{eq:J_numerical}.
In this case, the Hamiltonian dynamics defines then a separable system
and volume-preserving explicit methods can be used (see \cite{qin1993}).
More precisely, we define $\Phi^\xi_{ h}$ for all $x=(x_1,x_2)\in\R^2$
by $\Phi^\xi_{ h}(x)=(y_1,y_2)$, where,
\begin{align*}
  y_1^{1/2}&=x_1-\dfrac{ h}{2} \alpha \xi \dfrac{\partial V}{\partial x_1}(x)\\
  y_2&=x_2+ h \alpha \xi \dfrac{\partial V}{\partial x_2}((y_1^{1/2},x_2))\\
  y_1&=y_1^{1/2}-\dfrac{ h}{2} \alpha \xi \dfrac{\partial V}{\partial x_1}((y_1^{1/2},y_2))
\end{align*}
Thus, the non Lipschitz nonlinearity of $\nabla V$ is not an obstacle to the
well-posedness of this integrator.
\vspace{2mm}

\begin{figure}[!ht]
    \centering
    \includegraphics[width=0.70\textwidth]{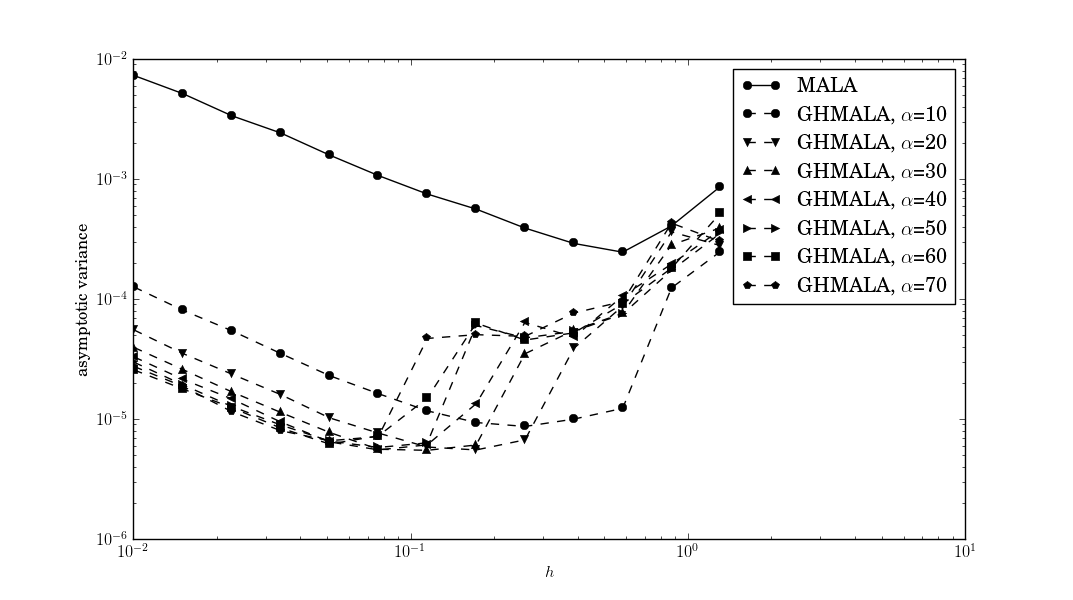}
    \caption{Variance comparison of MALA and GHMALA on the quartic Gaussian
    distribution}
    \label{fig:var_quartic}
\end{figure}
We plot in figure \ref{fig:var_quartic} the asymptotic variance for the time
average estimator build with GHMALA in the same way as for the warped Gaussian distribution.
We can observe that the decrease in variance between MALA and GHMALA for small
$ h$ is around 280.
Nevertheless, for larger $ h$, the explicit integration is not accurate
enough, which leads to an increase in the asymptotic variance for larger time steps.
Eventually, the smallest asymptotic variance of the time average estimator of GHMALA
is around 50 times lower than the smallest asymptotic variance for MALA.

\section*{Conclusion}
We presented a class of unbiased algorithm that enables to benefit from the
variance reduction of the nonreversible Langevin equations \eqref{eq:ovLangevin_nrev}
with respect to the reversible dynamics \eqref{eq:ovLangevin_rev}.
More precisely, we presented two variations of these algorithms.
The first one (GMALA) can be viewed as a lifting method, and more specifically as a
generalized Metropolis Hastings methods on a lifted state space.
The second one (GHMALA), similar to the first one, can be viewed as a
Generalized Hybrid Monte-Carlo method.

Numerical experimentations show that variance reductions (compared with classical
MALA) of several orders of magnitude can be achieved for potentials
concentrated on a lower dimensional submanifold.
We also expect these algorithms to perform better in the case of entropic barriers.
The main difficulty is, in the case of GMALA, to use a proposal that enables
to achieve a sufficiently high average acceptance ratio (to compete with MALA).
For example this can be done by using a mid-point discretization.
Even though this scheme is implicit, the computation of the Metropolis-Hastings
acceptance probability does not require the computation of the Hessian of $\log \pi$.
In the case of GHMALA, numerical experimentations show that the choice of a
suitable Hamiltonian integrator may lead to large improvements and computational
cost reduction compared with the mid-point method.

\vspace{2mm}
\paragraph{\textbf{Acknowledgement :}}
This work was supported by a public grant as part of the Investissement d'avenir project,
reference ANR-11-LABX-0056-LMH, LabEx LMH.

\bibliographystyle{amsplain}
\bibliography{bibliographie}

\end{document}